[1994/12/01]
\documentclass{ijmart}
\chardef\bslash=`\\ 

\usepackage{multirow}
\usepackage{enumitem} 
\usepackage{rotating}
 \usepackage{color}
\usepackage{array}

\newtheorem{thm}{Theorem}[section]
\newtheorem{cor}[thm]{Corollary}
\newtheorem{lem}[thm]{Lemma}

\newtheorem{conj}[thm]{Conjecture}

\theoremstyle{definition}

\newtheorem{rem}{Remark}[section]

\theoremstyle{remark}

\def\ZZ{\mathbb{Z}}
\def\RR{\mathbb{R}}

\newcommand{\eval}[2][\right]{\relax
  \ifx#1\right\relax \left.\fi#2#1\rvert}

\begin{document}
\title{Improved bounds on the diameter of lattice~polytopes}

\author[Antoine Deza]{Antoine Deza}
\address{McMaster University, Hamilton, Ontario, Canada}
\email{deza@mcmaster.ca} 

\author[Lionel Pournin]{Lionel Pournin}
\address{LIPN, Universit{\'e} Paris 13, Villetaneuse, France}
\email{lionel.pournin@univ-paris13.fr} 

\begin{abstract}
We show that the largest possible diameter $\delta(d,k)$ of a $d$-dimensional polytope whose vertices have integer coordinates ranging between $0$ and $k$ is at most $kd - \lceil2d/3\rceil$ when $k\geq3$. In addition, we show that $\delta(4,3)=8$. This substantiates the conjecture whereby $\delta(d,k)$ is at most $\lfloor(k+1)d/2\rfloor$ and is achieved by a Minkowski sum of lattice vectors.
\end{abstract}
\maketitle

\section{Introduction}

The convex hull of a set of points with integer coordinates is called a \emph{lattice polytope}. If all the vertices of a lattice polytope are drawn from $\{0,1,\dots,k\}^d$, it is referred to as a \emph{lattice $(d,k)$-polytope}. The diameter of a polytope $P$, denoted by $\delta(P)$, is the diameter of its graph. The quantity we are interested in is the largest possible diameter $\delta(d,k)$ of a lattice $(d,k)$-polytope.

At the end of the 1980's, Naddef~\cite{Naddef1989} showed that $\delta(d,1)=d$. A consequence of this result is that all lattice $(d,1)$-polytopes satisfy the Hirsch bound: their diameter is at most the number of their facets minus their dimension. While polytopes violating the Hirsch bound have been found by Santos~\cite{Santos2012}, many questions related with the diameter of polytopes, and more generally with the combinatorial, geometric, and algorithmic aspects of linear optimization remain open. Related recent results include the successive tightening by Todd~\cite{Todd2014} and Sukegawa~\cite{Sukegawa2016} of the upper bound on the diameter of polytopes due to Kalai and Kleitman~\cite{KalaiKleitman1992}, a counterexample to a continuous analogue of the polynomial Hirsch conjecture by Allamigeon, Benchimol, Gaubert, and Joswig~\cite{AllamigeonBenchimolGaubertJoswig2014}, and the validation that transportation polytopes satisfy the Hirsch bound by Borgwardt, De Loera, and Finhold~\cite{BorgwardtDeLoeraFinhold2016}. For additional related results, we refer the reader to~\cite{AllamigeonBenchimolGaubertJoswig2014,BonifasDiSummaEisenbrandHahnleNiemeier2014,BorgwardtDeLoeraFinhold2016,Santos2012,Sukegawa2016,Todd2014} and references therein.

The result of Naddef was generalized in the beginning of the 1990's by Kleinschmidt and Onn~\cite{KleinschmidtOnn1992} who proved that $\delta(d,k)\leq kd$. In a recent article, Del Pia and Michini~\cite{DelPiaMichini2016} strenghtened this bound to $\delta(d,k) \leq kd - \lceil d/2\rceil$ when $k\geq 2$, and showed that $\delta(d,2)=\lfloor3d/2\rfloor$.
Pursuing the approach introduced by Del Pia and Michini, we prove the following upper bound.
\begin{thm}\label{UB+}
$\delta(d,k) \leq kd - \lceil 2d/3\rceil$ when $k\geq 3$.
\end{thm}
We slightly refine Theorem \ref{UB+}.
\begin{thm}\label{UB++}
The following inequalities hold:
\begin{enumerate}[label=$(\roman*)$]
\item $\delta(d,k)\leq{kd-\lceil2d/3\rceil-(k-2)}$ when $k\geq4$ ,
\item $\delta(d,3)\leq\lfloor7d/3\rfloor-1$ when $d\not\equiv2\mod{3}$,
\item $\delta(d,3)\leq\lfloor7d/3\rfloor$ when $d\equiv2\mod{3}$.
\end{enumerate}
\end{thm}
Investigating the lower bound on $\delta(d,k)$, Deza, Manoussakis, and Onn~\cite{DezaManoussakisOnn2015} introduced the \emph{primitive lattice polytope} $H_1(d,p)$ as the Minkowski sum of the following set of lattice vectors:
$$
\{v\in\ZZ^d\,:\,\|v\|_1\leq p\,,\ \gcd(v)=1\,,\ v\succ 0\}\mbox{,}
$$
where $\gcd(v)$ is the largest integer dividing all the coordinates of $v$, and $v\succ0$ when the first non-zero coordinate of $v$ is positive. They  showed that, for any $k\leq{2d-1}$, there exists  a subset of the generators of  $H_1(d,2)$ whose Minkowski sum is, up to translation, a lattice $(d,k)$-polytope with diameter $\lfloor (k+1)d/2\rfloor$. As a consequence, they obtain the lower bound $\delta(d,k)\geq \lfloor (k+1)d/2\rfloor$  for  $k\leq 2d-1$, and  propose the following conjecture:
\begin{conj}[\cite{DezaManoussakisOnn2015}]\label{CC}
$\delta(d,k)$ is at most $\lfloor(k+1)d/2\rfloor$, and is achieved, up to translation, by a Minkowski sum of lattice vectors.
\end{conj}
The $2$-dimensional case had been previously studied in the early 1990's independently by Thiele~\cite{Thiele1991}, Balog and B\'ar\'any~\cite{BalogBarany1991}, and Acketa and \v{Z}uni\'{c}~\cite{AcketaZunic1995}. It can also be found in Ziegler's book~\cite{Ziegler1995} as Exercise 4.15. These results on $\delta(2,k)$ can be summarized as follows:
\begin{thm}[\cite{AcketaZunic1995,BalogBarany1991,DezaManoussakisOnn2015,Thiele1991}]\label{2D}
For any $k$, there exists a value of $p$ so that $\delta(2,k)$ is achieved, up to translation, by the Minkowski sum of a subset of the generators of $H_1(2,p)$. Moreover, for any $p$, and for $k=\sum_{i=1}^pi\phi(i)$, $\delta(2,k)$ is uniquely achieved, up to translation, by  $H_1(2,p)$, where $\phi$ denotes Euler's totient function. Thus, $\delta(2,k)=6(\frac{k}{2\pi})^{2/3}+O(k^{1/3}\log k)$.
\end{thm}
We obtain a previously unknown value of $\delta(d,k)$ as a consequence of Theorem~\ref{UB++} and of the lower bound on $\delta(d,k)$ provided in~\cite{DezaManoussakisOnn2015}:
\begin{cor}\label{coro}
$\delta(4,3)=8$. 
\end{cor}
All the values of $\delta(d,k)$ known so far are reported in Table~\ref{delta(d.k)}.

This paper is organized as follows. 
\begin{table}[b]
\begin{center}
\setlength{\tabcolsep}{0pt}
\begin{tabular}{c>{\centering\arraybackslash} p{3pt} <{}>{\centering\arraybackslash} p{7pt} <{}c>{\centering\arraybackslash} p{2pt} <{}>{\centering\arraybackslash} p{6pt} <{}>{\centering\arraybackslash} p{2pt} <{}>{\centering\arraybackslash} p{0.7cm} <{}>{\centering\arraybackslash} p{0.7cm} <{}>{\centering\arraybackslash} p{0.7cm} <{}>{\centering\arraybackslash} p{0.7cm} <{}>{\centering\arraybackslash} p{0.7cm} <{}>{\centering\arraybackslash} p{0.7cm} <{}>{\centering\arraybackslash} p{0.7cm} <{}>{\centering\arraybackslash} p{0.7cm} <{}>{\centering\arraybackslash} p{0.7cm} <{}>{\centering\arraybackslash} p{0.7cm} <{}}
& & & & & & \multicolumn{11}{c}{$k$}\\
\multicolumn{6}{c|}{} & & $1$ & $2$ & $3$ & $4$ & $5$ & $6$ & $7$ & $8$ & $9$ & $10$\\
\cline{3-17}
\multirow{6}{*}{\begin{sideways}{$d$}\end{sideways}}
& & & 1 & & \multicolumn{1}{c|}{} & &  1 & 1 & 1 & 1 & 1 & 1 & 1 & 1 & 1 & $\dots$ \\
& & & 2 & & \multicolumn{1}{c|}{} & & 2 & 3 & 4 & 4 & 5 & 6 & 6 & 7 & 8 & $\dots$ \\
& & & 3 & & \multicolumn{1}{c|}{} & & 3 & 4 & 6 & \ &  & & & & & \\
& & & 4 & & \multicolumn{1}{c|}{} & & 4 & 6 & 8 & & & & & & &\\
& & & \vdots & & \multicolumn{1}{c|}{} & & $\vdots$ & $\vdots$ & & & & & & & &\\
& & & $d$ & & \multicolumn{1}{c|}{} & & $d$ & $\left\lfloor\frac{3}{2}d\right\rfloor$ & & & & & & & &\\\\
\end{tabular}\caption{The largest possible diameter $\delta(d,k)$ of a lattice $(d,k)$-polytope.}\label{delta(d.k)}
\end{center}
\end{table}
In Section~\ref{Sec.ADLP.Prelim}, we prove slightly more general versions of two lemmas from~\cite{DelPiaMichini2016}. Theorems~\ref{UB+} and \ref{UB++} are proven in Section~\ref{Sec.ADLP.Main}. Their proof is done by induction on the dimension. Two lemmas that allow to proceed with the inductive step in these proofs are given in Section~\ref{Sec.ADLP.Theo}. We discuss the limitations of the approach in Section~\ref{Sec.ADLP.Conc}, and provide some perspectives for possible extensions of our results.
%
\section{Preliminary lemmas}\label{Sec.ADLP.Prelim}
Given two vertices $u$ and $v$ of a polytope $P$, we call $d(u,v)$ their distance in the graph of $P$. If $F$ is a face of $P$, we further call
$$
d(u,F)=\min\{d(u,v):v\in{F}\}\mbox{.}
$$

The coordinates of a vector $x\in\mathbb{R}^d$ will be denoted by $x_1$ to $x_d$, and its scalar product with a vector $y\in\mathbb{R}^d$ by $x\mathord{\cdot}y$. We first recall a lemma introduced by Del Pia and Michini, see Lemma 2 in~\cite{DelPiaMichini2016}:
%
\begin{lem}[\cite{DelPiaMichini2016}]\label{facet}
Consider a lattice $(d,k)$-polytope $P$. If $u$ is a vertex of $P$ and $c\in\mathbb{R}^d$ a vector with integer coordinates, then $d(u,F)\leq c\mathord{\cdot}u-\gamma$ where $\gamma=\min\{c\mathord{\cdot}x : x\in{P}\}$ and $F=\{x\in{P}: c\mathord{\cdot}x=\gamma\}$.
\end{lem}
Lemma~\ref{Lem.ADLP.0} is a generalization of Lemma 4 from~\cite{DelPiaMichini2016}:
\begin{lem}\label{Lem.ADLP.0}
Consider a lattice $(d,k)$-polytope $P$. If $I$ is a subset of $\{1,\dots d\}$ such that $l_i\leq{x_i}\leq{h_i}$ for all $x\in{P}$ and all $i\in{I}$, then
$$
\delta(P)\leq{\delta(d-|I|,k)+\sum\limits_{i\in I}(h_i-l_i)}\mbox{.}
$$
\begin{proof}
We use an induction on $|I|$. The statement is obviously true when $I$ is empty, and simplifies to that of Lemma 4 from~\cite{DelPiaMichini2016} when $|I|=1$.

Assume that, for some integer $n\geq1$, the statement holds when $|I|={n}$. Further assume that $|I|=n+1$. Consider an index $j\in{I}$ and respectively denote by $L_j$
 and by $H_j$ the intersections of $P$ with $\{x\in\RR^d:x_j=l_j\}$ and with $\{x\in\RR^d:x_j=h_j\}$. We can assume without loss of generality that $L_j$ and $H_j$ are both non-empty. Note that $L_j$ and $H_j$ are faces of $P$ and, possibly up to an affine transformation, lattice $(d-1,k)$-polytopes. By assumption, if $x$ belongs to either $L_j$ or $H_j$, then $l_i\leq{x_i}\leq{h_i}$ for all $i\in{I\setminus\{j\}}$. Therefore, by induction, the following inequality holds:
\begin{equation}\label{Lem.ADLP.0.eq.0}
\max\{\delta(L_j),\delta(H_j)\}\leq\delta(d-|I|,k)+\sum_{i\in{I\setminus\{j\}}}(h_i-l_i)\mbox{.}
\end{equation}

Since $P$ is a lattice polytope, $d(x,L_j)\leq{x_j-l_j}$ and $d(x,H_j)\leq{h_j-x_j}$ for any vertex $x$ of $P$. Thus, for any two vertices $u$ and $v$ of $P$, we either have the inequality $d(u,L_j)+d(v,L_j)\leq{h_j-l_j}$ (when $u_j+v_j\leq{h_j+l_j}$) or the inequality $d(u,H_j)+d(v,H_j)\leq{h_j-l_j}$ (when $u_j+v_j>h_j+l_j$). As a consequence,
\begin{equation}\label{Lem.ADLP.0.eq.1}
\delta(P)\leq\max\{\delta(L_j),\delta(H_j)\}+h_j-l_j\mbox{.}
\end{equation}

Combining inequalities (\ref{Lem.ADLP.0.eq.0}) and (\ref{Lem.ADLP.0.eq.1}) completes the proof. 
\end{proof}
\end{lem}
The following result is obtained by invoking Lemma~\ref{facet} for two vertices $u$ and $v$ of a lattice $(d,k)$-polytope $P$, with the same, well-chosen vector $c$.
\begin{lem}\label{Lem.ADLP.1}
Consider two vertices $u$ and $v$ of a lattice $(d,k)$-polytope $P$. If $I$ is a subset of $\{1,...,d\}$ with cardinality at most $3$ such that $u_i+v_i\leq{k}$ when $i\in{I}$,
then the following inequality holds:
$$
d(u,v)\leq\delta(d-|I|,k)+\sum\limits_{i\in I}(u_i+v_i)\mbox{.}
$$
\end{lem}
\begin{proof}
The statement is obviously true when $I$ is empty. Therefore, we assume that $1\leq|I|\leq3$ in the remainder of the proof.

Consider the vector $c$ of $\mathbb{R}^d$ such that $c_i$ is equal to $1$ if $i\in{I}$ and to $0$ otherwise. By Lemma~\ref{facet}, any vertex $x$ of $P$ satisfies
$$
d(x,F)\leq{c\mathord{\cdot}x-\gamma}\mbox{,}
$$
where $\gamma=\min\{c\mathord{\cdot}x:x\in{P}\}$ and $F=\{x\in{P}:c\mathord{\cdot}x=\gamma\}$.

Hence, if $u$ and $v$ are two vertices of $P$, then
\begin{equation}\label{Lem.ADLP.1.eq.0}
d(u,v)\leq\delta(F)+c\mathord{\cdot}(u+v)-2\gamma\mbox{.}
\end{equation}

Observe that, for any $x\in{F}$ and any $i\in{I}$, the following double inequality holds since the coordinates of $x$ are non-negative and since $c\mathord{\cdot}x=\gamma$:
\begin{equation}\label{Lem.ADLP.1.eq.1}
0\leq{x_i}\leq\gamma\mbox{.}
\end{equation}

According to Theorem 3.3 from~\cite{NaddefPulleyblank1984}, there exists an index $j\in\{1, ..., d\}$ such that the orthogonal projection $\bar{F}$ of $F$ on the hyperplane $\{x\in\RR^d:x_j=0\}$ satisfies $\delta(\bar{F})=\delta(F)$. Note that $\bar{F}$ is a lattice $(d-1,k)$-polytope and that (\ref{Lem.ADLP.1.eq.1}) still holds for any $x\in\bar{F}$ and any $i\in{I}$. Hence, applying Lemma~\ref{Lem.ADLP.0} to $\bar{F}$ and to the set of indices $I\setminus\{j\}$ results in the following upper bound:
$$
\delta(\bar{F})\leq\delta(d-1-|I\setminus\{j\}|,k)+(|I|-1)\gamma\mbox{.}
$$

Observe that $|I\setminus\{j\}|$ is either $|I|-1$ (if $j\in{I}$), or $|I|$ (if $j\not\in{I}$). In both cases, $\delta(d-1-|I\setminus\{j\}|,k)\leq\delta(d-|I|,k)$. As in addition, $F$ and $\bar{F}$ have the same diameter, the above upper bound on $\delta(\bar{F})$ yields
$$
\delta(F)\leq\delta(d-|I|,k)+(|I|-1)\gamma\mbox{,}
$$
which, combined with (\ref{Lem.ADLP.1.eq.0}), results in the following inequality:
\begin{equation}\label{Lem.ADLP.1.eq.2}
d(u,v)\leq\delta(d-|I|,k)+\sum\limits_{i\in I}(u_i+v_i)+(|I|-3)\gamma\mbox{.}
\end{equation}

As $\gamma\geq0$ and $|I|\leq3$, this completes the proof.
\end{proof}
A key ingredient for the inductive step of our main proof is the following.
\begin{rem}\label{Rem.ADLP.1}
Note that the term $(|I|-3)\gamma$ in the right-hand side of (\ref{Lem.ADLP.1.eq.2}) is negative if both $1\leq|I|\leq2$ and  the sum $\sum_{i\in{I}}x_i$ is non-zero for all $x\in{P}$. As a consequence, the inequality provided by Lemma~\ref{Lem.ADLP.1} is strict in this case.
\end{rem}

We now state a technical lemma that will be invoked twice in Section~\ref{Sec.ADLP.Theo}.

\begin{lem}\label{Lem.ADLP.1.5}
Let $u^0$, ..., $u^p$ be the vertices of a lattice $(2,k)$-polytope, labeled clockwise or counter-clockwise. If $u^p=(0,0)$ and $u^0-u^1$ is either $(1,0)$, $(0,1)$, or $(1,1)$, then $u^j_1+u^j_2+2\leq{u^{j-1}_1+u^{j-1}_2}$ whenever $2\leq{j}<p$.
\end{lem}
\begin{proof}
Note that the cone pointed at $u^p$ and formed by the incident edges is contained in the positive orthant. Assuming that $u^0-u^1$ is either $(1,0)$, $(0,1)$, or $(1,1)$, the corresponding cone pointed at $u^0$ is contained in a translation of the negative orthant.
\begin{figure}[b]
\begin{centering}
\includegraphics{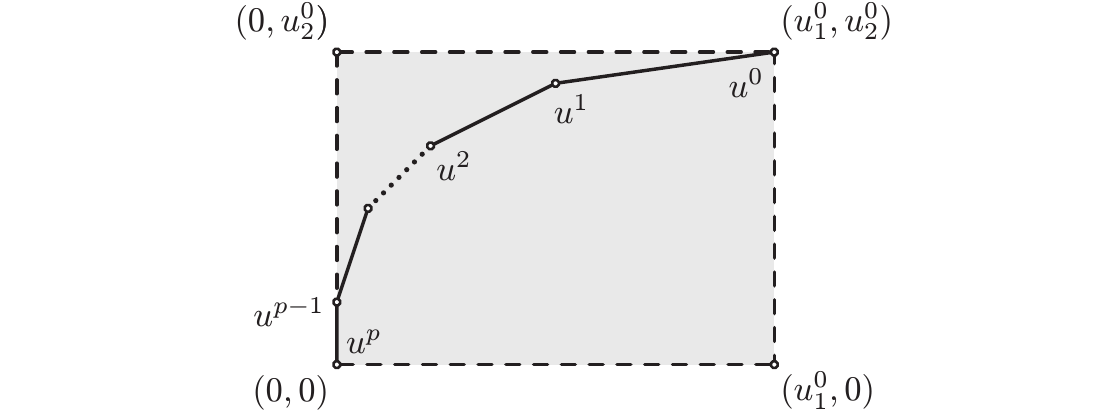}
\caption{A sketch of the lattice polygon with vertices $u^0$, ..., $u^p$.}\label{Fig.ADLP.1}
\end{centering}
\end{figure}
As a consequence, the polygon is inscribed in the rectangle $[0,u^0_1]\times[0,u^0_2]$. This situation is illustrated by Figure~\ref{Fig.ADLP.1} when the vertices of are labeled counter-clockwise.

Now observe that, by convexity, the only edges of the polygon that are possibly horizontal or vertical are incident to $u_0$ or to $u_p$. Hence, $u^j_1+1\leq{u^{j-1}_1}$ and $u^j_2+1\leq{u^{j-1}_2}$ for all $i\in\{2, ..., p-1\}$.
\end{proof}

\section{The inductive step}\label{Sec.ADLP.Theo}
The proof of Theorem~\ref{UB+} is done by induction on the dimension. The inductive step is split into two main cases, addressed by Lemmas~\ref{Lem.ADLP.3} and~\ref{Lem.ADLP.4}.
\begin{lem}\label{Lem.ADLP.3}
Let $P$ be a lattice $(d,k)$-polytope such that $d\geq3$ and $k\geq3$. Let $u$ and $v$ be two vertices of $P$ such that $u_i+v_i=k$ for all $i\in\{1,...,d\}$. If there exists a vertex $w$ adjacent to $u$ in the graph of $P$ such that  $w-u$ has at least two non-zero coordinates, then one of the following inequalities holds:
\begin{enumerate}[label=$(\roman*)$]
\item $d(u,v)\leq\delta(d-1,k)+k-1$,
\item $d(u,v)\leq\delta(d-2,k)+2k-2$,
\item $d(u,v)\leq\delta(d-3,k)+3k-2$.
\end{enumerate}
\end{lem}
\begin{proof}
Assume that there exists a vertex $w$ adjacent to $u$ in the graph of $P$ such that  $w-u$ has at least two non-zero coordinates. For any index $j\in\{1, ..., d\}$ such that $u_j\neq{w_j}$, we can require that $w_j<u_j$ by if needed, replacing $P$ by its symmetric with respect to the hyperplane $\{x\in\mathbb{R}^d:x_j=k/2\}$.

First assume that $u_j-w_j\geq2$ for some index $j\in\{1, ..., d\}$. In this case, $v_j+w_j\leq{k-2}$, and invoking Lemma~\ref{Lem.ADLP.1} with $I=\{j\}$ yields
$$
d(v,w)\leq\delta(d-1,k)+k-2\mbox{.}
$$

As $u$ and $w$ are adjacent in the graph of $P$, one then obtains $(i)$ from the triangle inequality. We therefore assume in the remainder of the proof that $0\leq{u_j-w_j}\leq1$ for all $j\in\{1, ..., d\}$.

Let $i_1$ and $i_2$ be distinct indices such that $u_{i_1}=w_{i_1}+1$ and $u_{i_2}=w_{i_2}+1$. Invoking Lemma~\ref{Lem.ADLP.1} with $I=\{i_1,i_2\}$ yields
\begin{equation}\label{Lem.ADLP.3.eq.1}
d(v,w)\leq\delta(d-2,k)+2k-2\mbox{.}
\end{equation}

According to Remark~\ref{Rem.ADLP.1}, if 
$$
F=\{x\in{P}:x_{i_1}+x_{i_2}=0\}
$$
is empty, then (\ref{Lem.ADLP.3.eq.1}) is strict. In this case, one obtains $(ii)$ from the triangle inequality because $u$ is adjacent to $w$ in the graph of $P$. In the sequel, we will further assume that $F$ is non-empty. In particular, $F$ is a non-empty face of $P$ of dimension at most $d-2$. Consider a sequence $u^0$, ..., $u^p$ of vertices of $P$ that forms a path from $u$ to $F$ in the graph of $P$. In other words, $u^0=u$, $u^p\in{F}$, and $u^{j-1}$ is adjacent to $u^j$ in the graph of $P$ whenever $0<j\leq{p}$. It can be assumed that for all $j\in\{1, ..., p\}$, the following inequality holds:
\begin{equation}\label{Lem.ADLP.3.eq.2}
u^j_{i_1}+u^j_{i_2}\leq{u^{j-1}_{i_1}+u^{j-1}_{i_2}-1}\mbox{.}
\end{equation}

For instance, such a path is provided by the simplex algorithm when minimizing $x_{i_1}+x_{i_2}$ from vertex $u$ under the constraint $x\in{P}$. It can also be required that $u^1=w$. Note that, because of this requirement, inequality (\ref{Lem.ADLP.3.eq.2}) is strict when $j=1$. Denote by $S_u$ the square made up of the points $x\in[0,k^d]$ so that $x_i=u^0_i$ whenever $i\in\{1,...,d\}\setminus\{i_1,i_2\}$. We will now review two cases depending on whether the path $u^0$, ..., $u^p$ remains in $S_u$ or not. In each case, we will prove that $(i)$, $(ii)$ or $(iii)$ holds.

Assume that the path $u^0$, ..., $u^p$ does not remain within $S_u$. In this case, there exists an index $i_3\in\{1,...,d\}\setminus\{i_1,i_2\}$ such that $u^r_{i_3}\neq{u^r_{i_3}}$ for some index $r\in\{1, ..., p\}$. Assume that $r$ is the smallest such index, or equivalently that vertices $u^0$ to $u^{r-1}$ all belong to $S_u$. As above, we can require that $u^r_{i_3}<u^0_{i_3}$ by if needed, replacing $P$ by its symmetric with respect to the hyperplane $\{x\in\mathbb{R}^d:x_{i_3}=k/2\}$. Recall that inequality (\ref{Lem.ADLP.3.eq.2}) holds whenever $1\leq{j}\leq{r}$, and is strict when $j=1$. As in addition, $u^r_{i_3}<u^0_{i_3}$, we have:
$$
\sum_{i\in{I}}(u_i^r+v_i)\leq3k-r-2\mbox{,}
$$
where $I=\{i_1,i_2,i_3\}$. Hence, by Lemma~\ref{Lem.ADLP.1},
$$
d(u^r,v)\leq\delta(d-3,k)+3k-r-2\mbox{.}
$$
As $d(u,u^r)$ is at most $r$, one obtains $(iii)$ from the triangle inequality.

Now assume that the path $u^0$, ..., $u^p$ remains within $S_u$. In this case, $u^0$ to $u^p$ are, up to an affine transformation, the vertices of a lattice $(2,k)$-polygon satisfying the requirements of Lemma~\ref{Lem.ADLP.1.5}. In particular, if $p\geq3$, then Lemma~\ref{Lem.ADLP.1.5} yields $u^2_{i_1}+u^2_{i_2}+2\leq{u^1_{i_1}+u^1_{i_2}}$. As a consequence,
$$
\sum_{i\in{I}}(u^2_i+v_i)\leq2k-4\mbox{,}
$$
where $I=\{i_1,i_2\}$, and by Lemma~\ref{Lem.ADLP.1},
$$
d(u^2,v)\leq\delta(d-2,k)+2k-4\mbox{.}
$$

As $d(u,u^2)\leq2$, one obtains $(ii)$ from the triangle inequality. We therefore assume that $p\leq2$ from now on.

Consider a sequence $v^0$, ..., $v^q$ of vertices of $P$ that forms a path from $v$ to $F$ in the graph of $P$. In other words, $v^0=v$, $v^q\in{F}$, and $v^{j-1}$ is adjacent to $v^j$ in the graph of $P$ whenever $0<j\leq{q}$. It can be required that for all $j\in\{1, ..., p\}$, the following inequality holds:
\begin{equation}\label{Lem.ADLP.3.eq.3}
v^j_{i_1}+v^j_{i_2}\leq{v^{j-1}_{i_1}+v^{j-1}_{i_2}-1}\mbox{,}
\end{equation}
by assuming, for instance, that this path is provided by the simplex algorithm when minimizing $x_{i_1}+x_{i_2}$ from vertex $v$ under the constraint $x\in{P}$. Denote by $S_v$ the square made up of the points $x\in[0,k^d]$ so that $x_i=v^0_i$ whenever $i\in\{1,...,d\}\setminus\{i_1,i_2\}$. We proceed as with sequence $u^0$, ..., $u^p$ and review two sub-cases depending on whether $v^0$, ..., $v^q$ all belong to $S_v$ or not.

Assume that vertices $v^0$, ..., $v^q$ do not all belong to $S_v$. In this case, there exists $i_3\in\{1,...,d\}\setminus\{i_1,i_2\}$ such that $v^r_{i_3}\neq{v^r_{i_3}}$ for some index $r\in\{1, ..., q\}$. Assume that $r$ is the smallest such index. In particular, vertices $v^0$ to $v^{r-1}$ all belong to $S_v$. We can again require that $v^r_{i_3}<v^0_{i_3}$ by if needed, replacing $P$ by its symmetric with respect to the hyperplane $\{x\in\mathbb{R}^d:x_{i_3}=k/2\}$.

As inequality (\ref{Lem.ADLP.3.eq.3}) holds whenever $1\leq{j}\leq{r}$, as $w_{i_1}+w_{i_2}\leq{k-2}$, and as $v^r_{i_3}<v^0_{i_3}$, we obtain the following:
$$
\sum_{i\in{I}}(v_i^r+w_i)\leq3k-r-3\mbox{,}
$$
where $I=\{i_1,i_2,i_3\}$. Therefore, Lemma~\ref{Lem.ADLP.1} yields:
$$
d(v^r,w)\leq\delta(d-3,k)+3k-r-3\mbox{.}
$$

Since $d(v,v^r)$ is at most $r$, and since $w$ is adjacent to $u$ in the graph of $P$, one obtains $(iii)$ from the triangle inequality.

Now assume that all the vertices $v^0$, ..., $v^q$ belong to $S_v$. Observe that if $v^0_{i_1}\geq{v^1_{i_1}+2}$ or $v^0_{i_2}\geq{v^1_{i_2}+2}$, then using $I=\{i_1\}$ in the former case and $I=\{i_2\}$ in the latter, Lemma~\ref{Lem.ADLP.1} immediately provides inequality $(i)$. We therefore assume that the differences $v^0_{i_1}-v^1_{i_1}$ and $v^0_{i_2}-v^1_{i_2}$ are both at most $1$. By (\ref{Lem.ADLP.3.eq.3}), the sum of these differences is at least $1$, and each of them must therefore be non-negative. In this case, $v^0$ to $v^q$ are, up to an affine transformation, the vertices of a lattice $(2,k)$-polygon satisfying the requirements of Lemma~\ref{Lem.ADLP.1.5}. In particular, if $q\geq3$, then Lemma~\ref{Lem.ADLP.1.5} yields $v^2_{i_1}+v^2_{i_2}+2\leq{v^1_{i_1}+v^1_{i_2}}$.

As a consequence,
$$
\sum_{i\in{I}}(v^2_i+w_i)\leq2k-5\mbox{,}
$$
where $I=\{i_1,i_2\}$, and by Lemma~\ref{Lem.ADLP.1},
$$
d(v^2,w)\leq{\delta(d-2,k)+2k-5}\mbox{.}
$$

As $d(v,v^2)\leq2$ and $d(u,w)=1$, inequality $(ii)$ is again obtained by using the triangle inequality, and we assume that $q\leq2$.

We have narrowed the possibilities to $p\leq2$ and $q\leq2$. Hence,
$$
d(u,v)\leq\delta(F)+4\mbox{.}
$$

As $F$ is a lattice $(d-2,k)$-polytope and as $k\geq3$, the right-hand side of this inequality is bounded above by $\delta(d-2,k)+2k-2$. Therefore, $(ii)$ holds.
\end{proof}

\begin{lem}\label{Lem.ADLP.4}
Let $P$ be a lattice $(d,k)$-polytope with $d\geq3$ and $k\geq3$. Let $u$ and $v$ be two vertices of $P$. If both $u$ and $v$ belong to $\{0,k\}^d$, and $u_i+v_i=k$ for all $i\in\{1,...,d\}$, then one of the following inequalities holds:
\begin{enumerate}[label=$(\roman*)$]
\item $d(u,v)\leq\delta(d-1,k)+k-1$,
\item $d(u,v)\leq\delta(d-2,k)+2k-2$,
\item $d(u,v)\leq\delta(d-3,k)+3k-2$.
\end{enumerate}
\end{lem}
\begin{proof} 
Assume that $u\in\{0,k\}^d$, $v\in\{0,k\}^d$, and $u_i+v_i=k$ whenever $1\leq{i}\leq{d}$. Consider an index $j\in\{1,...,d\}$. We can assume without loss of generality that $u_j=0$ and $v_j=k$ by, if needed, replacing $P$ by its symmetric with respect the the hyperplane $\{x\in\mathbb{R}^d:x_j=k/2\}$. Repeating this for all coordinates, we can therefore require that $u_i=0$ and $v_i=k$ for all $i\in\{1,...,d\}$.

Let $F=\{x\in{P}:x_1=0\}$. Observe that $d(v,F)\leq{k}$. This inequality is obtained, for instance, by invoking Lemma~\ref{facet} with the vector $c$ so that $c_i$ is equal to $1$ when $i=1$ and to $0$ otherwise. We will review three cases, depending on which vertices of $F$ are distance at most $k$ from $v$ in the graph of $P$.

First assume that there exists a vertex $w$ of $F$ so that $d(v,w)\leq{k}$ and $w$ has at least two coordinates distinct from $k$ other than $w_1$. Let $i_1$ and $i_2$ be two distinct indices in $\{2, ..., d\}$ so that $w_{i_1}<k$ and $w_{i_2}<k$. Let $G=\{x\in{F}:x_{i_1}+x_{i_2}=0\}$. In this case,
$$
\sum_{i\in{I}}(u_i+w_i)\leq2k-2\mbox{,}
$$
where $I=\{1,i_1,i_2\}$. Hence, by Lemma~\ref{Lem.ADLP.1},
$$
d(u,w)\leq\delta(d-3,k)+2k-2\mbox{.}
$$

As $d(v,w)\leq{k}$, using the triangle inequality provides $(iii)$.

Now assume that there exists a vertex $w$ of $F$ so that $d(v,w)\leq{k}$ and $w$ has exactly one coordinate distinct from $k$ other than $w_1$. Let $j\in\{2, ..., d\}$ be an index so that $w_j<k$. We consider two sub-cases depending on the value of $w_j$. First assume that $w_j\leq{k-2}$. In this case, one obtains the following inequality by invoking Lemma~\ref{Lem.ADLP.1} with $I=\{j\}$:
$$
d(u,w)\leq\delta(d-2,k)+k-2\mbox{,}
$$

As $d(v,w)\leq{k}$, the triangle inequality then provides $(ii)$ because $d(v,w)\leq{k}$. Now assume that $w_j=k-1$. In this case, consider face $G$ of $P$ made up of all the points $x\in{P}$ so that $x_i=k$ when $i\in\{2, ..., d\}\setminus\{j\}$. Note that $G$ is at most $2$-dimensional and at least $1$-dimensional because it contains both $v$ and $w$. In other words, $G$ is either an edge of $P$, or one of its polygonal faces.

Since $v_j=k$ and $w_j=k-1$, $v$ and $w$ necessarily have distance at most $2$ in the graph of $G$. Indeed, either they are adjacent in this graph, or there exists a unique vertex $x$ of $G$, such that $x_j=k$ and $1\leq{x_1}<k$. There cannot be another such vertex because it would be collinear with $x$ and $v$. The vertices of $G$ adjacent to $x$ are then $v$ and $w$, and their distance is at most $2$.

As a consequence,
$$
d(u,v)\leq\delta(d-1,k)+2\mbox{.}
$$

Since $k\geq3$, inequality $(i)$ follows.

Finally, assume that the unique vertex $w$ of $F$ such that $d(v,w)\leq{k}$ satisfies $w_1=0$ and $w_i=k$ when $2\leq{i}\leq{d}$. In this case, the segment with vertices $v$ and $w$ is an edge of $P$. Hence, $d(v,F)=1$ and $d(u,v)\leq\delta(d-1,k)+1$. As $k\geq3$, inequality $(i)$ holds, which completes the proof.
\end{proof}

Combining Lemmas~\ref{Lem.ADLP.3} and~\ref{Lem.ADLP.4}, one obtains Theorem~\ref{Theo} that provides the inductive step for the proof of Theorem~\ref{UB+}:

\begin{thm}\label{Theo}
Assume that $d\geq3$ and $k\geq3$. If $u$ and $v$ are two vertices of a lattice $(d,k)$-polytope $P$, then one of the following inequalities holds:
\begin{enumerate}[label=$(\roman*)$]
\item $d(u,v)\leq\delta(d-1,k)+k-1$,
\item $d(u,v)\leq\delta(d-2,k)+2k-2$,
\item $d(u,v)\leq\delta(d-3,k)+3k-2$.
\end{enumerate}
\end{thm}
\begin{proof}
Consider two vertices $u$ and $v$ of a lattice $(d,k)$-polytope $P$. Note that, if $u_j+v_j\neq{k}$ for some index $j\in\{1,...,d\}$, then we can assume without loss of generality that $u_j+v_j<k$ by, if needed, replacing $P$ by its symmetric with respect the the hyperplane $\{x\in\mathbb{R}^d:x_j=k/2\}$. In this case, invoking Lemma~\ref{Lem.ADLP.1} with $I=\{j\}$ provides inequality $(i)$. In the remainder of the proof we will assume that $u_i+v_i=k$ whenever $1\leq{i}\leq{d}$.

Assume that $0<u_i<k$ for some index $i\in\{0,..., d\}$. If $x_i\geq{u_i}$ for all $x\in{P}$, then, invoking Lemma~\ref{Lem.ADLP.0} with $I=\{i\}$, provides $(i)$. By tLemma~\ref{Lem.ADLP.0}, $(i)$ also holds when $x_i\leq{u_i}$ for all $x\in{P}$. Hence we can assume that there exist two vertices adjacent to $u$ in the graph of $P$ whose $i$-th coordinates are respectively less and greater than $u_i$. As argued in~\cite{DelPiaMichini2016}, there exists an index $j\in\{1,...,d\}$ distinct from $j$ so that one of these two vertices has a $j$-th coordinate distinct $u_j$. Indeed, $u$ would otherwise be contained in the segment bounded by these vertices. In this case, the result follows from Lemma~\ref{Lem.ADLP.3}.

By the same argument, the desired result also holds when $0<v_i<k$ for some index $i\in\{0,..., d\}$. Finally, if $u$ and $v$ both belong to $\{0,k\}^d$, then Theorem~\ref{UB+}is a direct consequence of Lemma~\ref{Lem.ADLP.4}.
\end{proof}

\section{Proofs of Theorems~\ref{UB+} and~\ref{UB++}}\label{Sec.ADLP.Main}
We first prove Theorem~\ref{UB+} by induction.
\begin{proof}[Proof of Theorem~\ref{UB+}]
Assume that $k\geq3$. According to Lemma~\ref{Lem.ADLP.0},
$$
\delta(d,k)\leq\delta(1,k) + (d-1)k\mbox{.}
$$

Since $\delta(1,k)=1$, this can be rewritten as
$$
\delta(d,k)\leq{kd}-(k-1)\mbox{.}
$$

As $\lceil{2d/3}\rceil\leq{k-1}$ when $k\geq3$ and $d\leq3$, this inequality yields the desired bound on $\delta(d,k)$ when $d\leq3$.

Assume that $d\geq4$ and $\delta(d-p,k)\leq{k(d-p)-\lceil{2d/3}\rceil}$ whenever $1\leq{p}\leq3$. Consider two vertices $u$ and $v$ of a lattice $(d,k)$-polytope $P$ whose distance in the graph of $P$ is precisely $\delta(d,k)$. By Theorem~\ref{Theo},
$$
d(u,v)\leq \delta(d-p,k)+pk-q\mbox{,}
$$
where $1\leq{p}\leq3$ and $q$ is equal to $1$ when $p=1$ and to $2$ otherwise.

Hence, by induction,
$$
d(u,v)\leq kd-\lceil{2(d-p)/3}\rceil-q\mbox{,}
$$

As $2p/3\leq{q}$ for all the pairs $(p,q)$ considered in this proof, it follows that $d(u,v)\leq kd-{2d/3}$. Since $d(u,v)$ is equal to $\delta(d,k)$ and since these are integer quantities, the desired bound on $\delta(d,k)$ holds.
\end{proof}
%

Theorem~\ref{UB++} relies on the same induction than Theorem~\ref{UB+}. The only difference lies in the way this induction is initialized. 

\begin{proof}[Proof of Theorem~\ref{UB++}]
We first prove assertion $(i)$. Assume that $k\geq 4$. Since $\delta(1,k)=1$, this assertion holds when $d=1$. According to Theorem~\ref{2D}, $\delta(2,k)\leq k$ when $k\geq4$; that is, assertion $(i)$ also holds when $d=2$. By Lemma~\ref{Lem.ADLP.0}, $\delta(3,k)\leq\delta(2,k)+k$. As a consequence, $\delta(3,k)\leq 2k$ when $k\geq 4$. In other words,  assertion $(i)$ further holds when $d=3$. Using Theorem~\ref{Theo} inductively then provides $(i)$ for any $d$.

Now assume that $k=3$ and note that $\delta(1,3)=1$, $\delta(2,3)=4$, and $\delta(3,3)=6$ (see Table~\ref{delta(d.k)}). Consider two vertices  $u$ and $v$ of a lattice $(4,3)$-polytope $P$ such that $d(u,v)=\delta(4,3)$. Invoking Theorem~\ref{Theo} with $d=4$ and $k=3$ yields $\delta(4,3)\leq 8$. Thus, assertions $(ii)$ and $(iii)$ both hold when $d\leq 4$.  Theorem~\ref{Theo} can then be used inductively to prove assertions $(ii)$ and $(iii)$ for any $d$.
\end{proof}

%

\section{Discussion}\label{Sec.ADLP.Conc}
Observe that the term $d/2$ in the bound by Del Pia and Michini, and the term $2d/3$ in our bound are both derived from the expression $(|I|-1)d/|I|$, where $I$ is the set in the statement of Lemma~\ref{Lem.ADLP.1}. The former bound is obtained with $|I|=2$ and the latter with $|I|=3$. A first limitation of the approach is that Lemma~\ref{Lem.ADLP.1} can only be used up to $|I|=3$.  Another limitation comes from Lemma~\ref{Lem.ADLP.1.5} that only deals with lattice polygons. In order to further improve the result obtained with this approach, a similar lemma regarding $3$-dimensional lattice polytopes may be needed.

Table~\ref{delta(d.k)} suggests that the next values of $\delta(d,k)$ to determine are $\delta(d,3)$ when $d\geq5$ and $\delta(3,k)$ when $k\geq4$. One may be able to compute $\delta(3,4)$, $\delta(3,5)$, and $\delta(5,3)$ for which the known lower and upper bounds differ by only one. More precisely $7\leq\delta(3,4)\leq8$, $9\leq\delta(3,5)\leq10$, and $10\leq\delta(5,3)\leq11$. In these three cases, the computational search space can be significantly limited by using the following necessary conditions for the upper bound to be achieved by a given lattice $(d,k)$-polytope $P$:
\begin{enumerate}[label=$(\roman*)$]
\item If $u$ and $v$ are two vertices of $P$ such that $\delta(u,v)=\delta(P)$, then $u_i+v_i=k$ whenever $1\leq{i}\leq{d}$, and the differences between these vertices and their neighbors in the graph of $P$ belong to $\{-1,0,1\}^d$,
\item The intersection of $P$ with any facet of the cube $[0,k]^d$ is, up to an affine transformation, a lattice $(d-1,k)$-polytope of diameter $\delta(d-1,k)$.
\end{enumerate}

Observe that these conditions could also be used for a possible inductive proof of Conjecture \ref{CC} when $k=3$, that is $\delta(d,3)=2d$.

\bibliographystyle{ijmart}
\bibliography{LatticePolytopes}

\end{document}